\documentclass{amsart}
\usepackage{txfonts}
\usepackage{}

\newtheorem{theorem}{Theorem}[section]
\newtheorem{lemma}[theorem]{Lemma}
\newtheorem{corollary}[theorem]{Corollary}
\newtheorem{proposition}[theorem]{Proposition}

\theoremstyle{definition}

\theoremstyle{remark}
\newtheorem{remark}[theorem]{Remark}

\numberwithin{equation}{section}

\begin{document}

\title{On compact Riemannian manifolds with harmonic Weyl curvature}

\author{Hai-Ping Fu}
\address{Department of Mathematics,  Nanchang University, Nanchang 330031, P. R. China}
\email{mathfu@126.com}
\thanks{Supported by National Natural Science Foundations of China \#11761049,  Jiangxi Province
Natural Science Foundation of China \#20171BAB201001.}

\author{Hui-Ya He*}
\address{Department of Mathematical Sciences, Tsinghua University, Beijing 100084, P. R. China}
\email{hhy15@mails.tsinghua.edu.cn}





\begin{abstract}
 We give some rigidity theorems for an n$(\geq4)$-dimensional  compact Riemannian manifold with harmonic Weyl curvature, positive scalar curvature and positive constant $\sigma_2$. Moreover, when $n=4,$ we prove that a 4-dimensional compact locally conformally flat Riemannian manifold with positive scalar curvature and positive constant $\sigma_2$ is isometric to a quotient of the round $\mathbb{S}^4$.
\end{abstract}

\maketitle

\section{Introduction and main results}
Recall that an $n$-dimensional Riemannian manifold $(M^n, g)(n\geq4)$  is said to be  a manifold with harmonic Weyl curvature (resp. harmonic curvature)
if the divergence of its Weyl curvature tensor $W$ vanishes, i.e., $\delta W=0$ (resp. the divergence of its Riemannian curvature tensor $Rm$ vanishes, i.e., $\delta Rm=0$).
In view of the second Bianchi identity, we know that $M^n$ has harmonic Weyl curvature (resp. harmonic
curvature) if and only if the Schouten tensor (resp. the Ricci tensor) of $M^n$ is a Codazzi tensor. By the Bianchi identity, the scalar curvature on manifolds with harmonic curvature is constant.  $M^n$ has harmonic  curvature if and only if $M^n$ has harmonic Weyl curvature and constant scalar curvature. Thus, every Riemannian manifold with parallel Ricci tensor has harmonic
curvature. The locally conformally
flat Riemannian  manifolds and the products of Einstein manifolds are
also important examples of manifolds with harmonic Weyl curvature, however, the
converse does not hold (see \cite{{B}}, for example). According to the decomposition of the Riemannian curvature tensor, the metric with harmonic Weyl curvature is a natural candidate for this study  since one of the important problems in Riemannian geometry is to understand classes
of metrics that are, in some sense, close to being Einstein or having constant curvature.

 The complete manifolds  with
harmonic curvature have been studied in literature (e.g., \cite{{C},{Ca},{D},{F},{F2},{FP},{FX},{FX2},{HV},{IS},{K},{LZ},{PRS},{S},{Ta},{XZ}}).
Some isolation theorems of Weyl curvature tensor of positive Einstein manifolds are given in \cite{{Ca2},{FX2},{HV},{IS},{S}}, when its $L^{p}$-norm is small.
Some scholars  classify  conformally flat  manifolds  satisfying some curvature $L^{p}$-pinching conditions \cite{{Ca},{FP},{FX},{HV},{PRS},{XZ}}. Recently,
Tran \cite{T} obtain two rigidity results for a closed Riemannian manifold
with harmonic Weyl curvature, which are  a generalization of Tachibana's theorem for non-negative curvature operator \cite{Ta} and  integral gap result which extends Theorem 1.10  for manifolds with harmonic curvature in \cite{F2}. We are interested in some
pinching problems for compact Riemannian manifold  with harmonic Weyl curvature.

Now we  introduce the definition of the Yamabe constant. Given a compact  Riemannian manifold $(M^n, g)(n\geq3)$, we consider the Yamabe functional
$$Q_g\colon C^{\infty}_{+}(M^n)\rightarrow\mathbb{R}\colon f\mapsto Q_g(f)=\frac{\frac{4(n-1)}{n-2}\int_{M^n}|\nabla f|^2\mathrm{d}v_g+\int_{M^n} Rf^2\mathrm{d}v_g}{(\int_{M^n} f^{\frac{2n}{n-2}}\mathrm{d}v_g)^{\frac{n-2}{n}}},$$
where $R$ denotes the   scalar curvature of $M^n$.
It follows that $Q_g$ is bounded below by H\"{o}lder inequality. We set
$$Y(M^n, [g])=\inf\{Q_g(f)|f\in C^{\infty}_{+}(M^n)\}.$$
This constant $Y(M^n,[g])$ is an invariant of the conformal class of $(M^n, g)$, called the Yamabe constant.
The important works of Aubin, Schoen, Trudinger and Yamabe showed that the infimum in $\inf\{Q_g(f)|f\in C^{\infty}_{+}(M^n)\}$ is always achieved (see \cite{{A},{LP}}).
The Yamabe constant of a given compact manifold
is determined by the sign of scalar curvature \cite{A}. For a Riemannian manifold $(M^n, g)(n\geq3)$, we denote by $\sigma_2(A_g)$ the second elementary symmetric
function of the eigenvalues of the so-called Schouten tensor $A_g=Ric-\frac{R}{2(n-1)}g$ with respect to $g$. Hence
\begin{equation}2\sigma_2(A_g)=(trA_g)^2-|A_g|^2=\frac{(n-2)^2}{4n(n-1)}R^2-|\mathring{Ric}|^2,\end{equation}
where $\mathring{Ric}$ denotes the trace-free Ricci tensor.

In this note, we obtain the following rigidity theorems.

\begin{theorem}
Let $(M^n, g)(n\geq4)$  be an $n$-dimensional   compact Riemannian manifold  with harmonic Weyl curvature, positive scalar curvature and  positive constant $\sigma_2(A_g)$. If
\begin{equation}
|W|^2+\frac{2n^2(n-4)}{(n-2)^3}|\mathring{Ric}|^{2}
<\frac{16n}{{(n-2)^3}}\sigma_2(A_g),
\end{equation}
then $(M^n, g)$ is an Einstein manifold.
\end{theorem}

\begin{remark} The inequality (1.2) of this theorem is optimal.
The critical case is given by the product manifold $\mathbb{S}^1\times \mathbb{S}^{n-1}$.
\end{remark}

\begin{theorem}
Let $(M^n, g)(n\geq4)$  be an $n$-dimensional   compact Riemannian manifold  with harmonic Weyl curvature, positive scalar curvature and  positive constant $\sigma_2(A_g)$. If
\begin{equation}
|W|^2+\left[\frac{2n}{n-2}-\frac{16(n-1)}{n(n-2)^2C(n)^2}\right]|\mathring{Ric}|^{2}
<\frac{32(n-1)}{n(n-2)^2C(n)^2}\sigma_2(A_g),
\end{equation}
where\begin{equation*}C(n)=
\begin{cases}\frac{\sqrt{6}}{2}, \quad \text{for} \ n=4, \\\frac{8}{\sqrt{10}}, \quad \text{for} \ n=5, \\\frac{2(n-2)}{\sqrt{n(n-1)}}+ \frac{n^2-n-4}{\sqrt{(n-2)(n-1)n(n+1)}}, \quad \text{for} \  n\geq6,
\end{cases}
\end{equation*}
then $(M^n, g)$ is isometric to a quotient of the round $\mathbb{S}^n$.
\end{theorem}
\begin{corollary}\label{66666}
Let $(M^4, g)$  be a $4$-dimensional   compact Riemannian manifold  with harmonic Weyl curvature, positive scalar curvature and  positive constant $\sigma_2(A_g)$. If
\begin{equation}
|W|^2<8\sigma_2(A_g),
\end{equation}
then $(M^4, g)$ is isometric to a quotient of the round $\mathbb{S}^4$ or a  $\mathbb{CP}^2$ with the Fubini-Study metric.
\end{corollary}

Since for a  $4$-dimensional  locally conformally flat Riemannian manifold $W\equiv 0$, we can directly get the following corollary from Corollary \ref{66666}.
\begin{corollary}\label{666666}
Let $(M^4, g)$  be a $4$-dimensional   compact locally conformally flat Riemannian manifold  with  positive scalar curvature and  positive constant $\sigma_2(A_g)$.
Then $(M^4, g)$ is isometric to a quotient of the round $\mathbb{S}^4$.
\end{corollary}
\begin{remark} Hu-Li-Simon \cite{HLS} proved this similar result in which the Schouten tensor $A_g$ is semi-positive definite.
\end{remark}
\begin{corollary}\label{6666666}
Let $(M^4, g)$  be a $4$-dimensional   compact locally conformally flat Riemannian manifold  with  positive Yamabe constant and  $\int_{M^4}\sigma_2(A_g)>0$.
Then $(M^4, g)$ is conformal to a quotient of the round $\mathbb{S}^4$.
\end{corollary}
\begin{remark} This result had been proved by Gursky \cite{G}, Xiao and the first author \cite{FX}.
\end{remark}
\begin{corollary}\label{66666666}
Let $(M^n,g)(n\geq5)$ be an n-dimensional compact locally conformally flat Riemannian  manifold with positive scalar curvature and positive constant $\sigma_2(A_g)$. If
\begin{equation}|\mathring{Ric}|^2<\frac{1}{n(n-1)}R^2,\end{equation}
then $(M^n, g)$  is isometric to a quotient of the round $\mathbb{S}^n$.
\end{corollary}
\begin{remark}
In \cite{FX}, Xiao and the first author proved that for a compact locally conformally flat Riemannian manifold $(M^n,g)(n\geq4)$ with  positive constant scalar curvature, if $\left(\int_{M^n}|\mathring{Ric}|^{\frac n2}\right)^{\frac 2n}<\frac{1}{n(n-1)}Y(M^n, [g])$,  then $(M^n,g)$ is isometric to a quotient of the round $\mathbb{S}^n$.
In \cite{PRS}, Pigola, Rigoli and Setti showed that for a complete locally conformally flat Riemannian  manifold $(M^n,g)(n\geq4)$ with  positive constant scalar curvature, if $|\mathring{Ric}|^2<\frac{1}{n(n-1)}R^2$, $(M^n,g)$ is isometric to a quotient of the round $\mathbb{S}^n$.
\end{remark}

\begin{theorem}\label{5555555}
Let $(M^n, g)(n\geq5)$  be an $n$-dimensional   compact Riemannian manifold  with harmonic Weyl curvature and  positive Yamabe constant. If
\begin{equation}
\begin{split}C(n)\left(\int_{M^n}|W|^{\frac n2}\right)^{\frac 2n}+2\sqrt{\frac{n-1}{n}}\left(\int_{M^n}|\mathring{Ric}|^{\frac n2}\right)^{\frac 2n}
<D(n)Y(M^n,[g]),
\end{split}
\end{equation}
where \begin{equation*}D(n)=
\begin{cases}\frac{3}{8}, \quad \text{for} \ n=5, \\ \frac{2}{n}, \quad  \text{for} \  n\geq6,
\end{cases}\end{equation*} then $(M^n, g)$ is locally conformally flat. Moreover, when  $n=5$,
the same result holds only assuming the weak inequality.
\end{theorem}
\begin{remark} This result is a straightforward reworking
of the first author's arguments on Theorem 1.10 in \cite{F2}. Although this result has been proved in \cite{T}, the pinching condition (1.6) is better for $n\geq 6$. When $n=4$, this result has been showed by the first author \cite{F}.
\end{remark}

{Acknowledgement:} The authors are very grateful to Professor Haizhong Li for his guidance and constant support.


\section{Proofs of Lemmas}
In this section, in order to prove some main theorems in this article, we first give the following lemmas.

In what follows, we adopt, without further comment, the moving frame notation with respect to a chosen local orthonormal frame. And we adopt the Einstein convention.

Let $M^n(n\geq3)$ be an  $n$-dimensional Riemannian manifold with harmonic Weyl curvature.
Decomposing the Riemannian curvature tensor $Rm$ into irreducible components yields
\begin{eqnarray*}
R_{ijkl}&=&W_{ijkl}+\frac{1}{n-2}(R_{ik}\delta_{jl}-R_{il}\delta_{jk}+R_{jl}\delta_{ik}-R_{jk}\delta_{il})\nonumber\\
&&-\frac{R}{(n-1)(n-2)}(\delta_{ik}\delta_{jl}-\delta_{il}\delta_{jk})\nonumber\\
&=&W_{ijkl}+\frac{1}{n-2}(\mathring{R}_{ik}\delta_{jl}-\mathring{R}_{il}\delta_{jk}+\mathring{R}_{jl}\delta_{ik}-\mathring{R}_{jk}\delta_{il})\nonumber\\
&&+\frac{R}{n(n-1)}(\delta_{ik}\delta_{jl}-\delta_{il}\delta_{jk})\nonumber\\
&=&W_{ijkl}+\frac{1}{n-2}(A_{ik}\delta_{jl}-A_{il}\delta_{jk}+A_{jl}\delta_{ik}-A_{jk}\delta_{il}),
\end{eqnarray*}
where $R$  is the scalar curvature, $R_{ijkl}$, $W_{ijkl}$, $R_{ij}$, $\mathring{R}_{ij}$  and $A_{ij}$  denote the components of $Rm$, the Weyl curvature tensor $W$,  the Ricci tensor $Ric$, the trace-free Ricci tensor $\mathring{Ric}=Ric-\frac{R}{n}g$  and  the Schouten tensor $A_g=Ric-\frac{R}{2(n-1)}g$, respectively.

By the second Bianchi identity, we have
\begin{equation}\mathring{R}_{ij,k}-\mathring{R}_{ik,j}={R}_{likj,l}+\frac{R_{,j}}{n}\delta_{ik}-\frac{R_{,k}}{n}\delta_{ij},\end{equation}
and
\begin{equation}\mathring{R}_{ij,j}=\frac{n-2}{2n}R_{,i}.\end{equation}
Moreover, by the assumption that Weyl curvature is harmonic,
we compute
\begin{equation*}
\begin{split}
0&=W_{ijkl,l}={R}_{ijkl,l}-\frac{1}{n-2}(\mathring{R}_{ik,l}\delta_{jl}-\mathring{R}_{il,l}\delta_{jk}+\mathring{R}_{jl,l}\delta_{ik}-\mathring{R}_{jk,l}\delta_{il})
-\frac{R_{,l}}{n(n-1)}(\delta_{ik}\delta_{jl}-\delta_{il}\delta_{jk})\\
&=\frac{n-3}{n-2}(\mathring{R}_{ki,j}-\mathring{R}_{kj,i})-\frac{n-3}{2n(n-1)}(R_{,i}\delta_{kj}-R_{,j}\delta_{ki}).
\end{split}
\end{equation*}
Thus we  get
\begin{equation}
\mathring{R}_{ki,j}-\mathring{R}_{kj,i}=\frac{n-2}{2n(n-1)}(R_{,i}\delta_{kj}-R_{,j}\delta_{ki}).
\end{equation}
Combining (2.3) with the second Bianchi identity, we obtain
\begin{equation}
{W}_{ijkl,h}+{W}_{ijlh,k}+{W}_{ijhk,l}=0.
\end{equation}

\begin{lemma}
Let $M^n(n\geq 4)$  be an $n$-dimensional complete  Riemannian manifold
with harmonic Weyl curvature.   Then
\begin{equation}
\frac12\triangle|\mathring{Ric}|^2=|\nabla \mathring{Ric}|^2+\frac{n-2}{2(n-1)}\mathring{R}_{ij}R_{,ij}+W_{kijl}\mathring{R}_{ij}\mathring{R}_{kl}+\frac{n}{n-2}\mathring{R}_{ij}\mathring{R}_{jl}\mathring{R}_{li}
+\frac{R}{n-1}|\mathring{Ric}|^2.
\end{equation}
\end{lemma}
\begin{proof}
We compute
\begin{equation*}
\frac12\triangle|\mathring{Ric}|^2=|\nabla \mathring{Ric}|^2+\langle \mathring{Ric}, \triangle \mathring{Ric}\rangle=|\nabla \mathring{Ric}|^2+\mathring{R}_{ij}\mathring{R}_{ij,kk}.
\end{equation*}
By Ricci identities and (2.3), we obtain
\begin{equation*}
\begin{split} \mathring{R}_{ij,kk}&=\mathring{R}_{ik,jk}+\frac{n-2}{2n(n-1)}(R_{,jk}\delta_{ik}-R_{,kk}\delta_{ij})\\
&=\mathring{R}_{ki,kj}+\mathring{R}_{li}R_{lkjk}+\mathring{R}_{kl}R_{lijk}+\frac{n-2}{2n(n-1)}(R_{,ji}-R_{,kk}\delta_{ij})\\
&=\mathring{R}_{kk,ij}+\frac{n-2}{2n(n-1)}(R_{,ij}\delta_{kk}-R_{,kj}\delta_{ki})+\mathring{R}_{li}R_{lj}+\mathring{R}_{kl}R_{lijk}+\frac{n-2}{2n(n-1)}(R_{,ji}-R_{,kk}\delta_{ij})\\
&=\frac{n-2}{2n(n-1)}(nR_{,ij}-R_{,kk}\delta_{ij})+\mathring{R}_{li}\mathring{R}_{lj}+\mathring{R}_{kl}R_{lijk}+\frac{R}{n}\mathring{R}_{ji},
\end{split}
\end{equation*}
which gives
\begin{equation*}
\begin{split}
\frac12\triangle|\mathring{Ric}|^2&=|\nabla \mathring{Ric}|^2+\frac{n-2}{2(n-1)}\mathring{R}_{ij}R_{,ij}-\frac{n-2}{2n(n-1)}R_{,kk}\delta_{ij}\mathring{R}_{ij}
+\mathring{R}_{ij}\mathring{R}_{li}\mathring{R}_{lj}+\mathring{R}_{ij}\mathring{R}_{kl}R_{lijk}+\frac{R}{n}|\mathring{Ric}|^2\\
&=|\nabla \mathring{Ric}|^2+\frac{n-2}{2(n-1)}\mathring{R}_{ij}R_{,ij}+W_{kijl}\mathring{R}_{ij}\mathring{R}_{kl}+\frac{n}{n-2}\mathring{R}_{ij}\mathring{R}_{jl}\mathring{R}_{li}
+\frac{R}{n-1}|\mathring{Ric}|^2.
\end{split}
\end{equation*}
This completes the proof of this Lemma.
\end{proof}

\begin{lemma}
Let $M^n(n\geq 4)$  be an $n$-dimensional complete  Riemannian manifold
with harmonic Weyl curvature.      Then
\begin{equation}
\frac12\triangle|W|^2=|\nabla W|^2+2W_{ijkl}(2W_{hjkm}W_{hilm}
-\frac 12W_{ijhm}W_{klhm})+\frac{2R}{n}|W|^2+2W_{ijkl}W_{ijkh}\mathring{R}_{hl}.
\end{equation}
\end{lemma}
\begin{remark} (2.6) is a straightforward reworking
of the first author's arguments on Lemma  2.4 in \cite{F2}. When $n=4$, it follows that $W_{ijkl}W_{ijkh}=\frac14|W|^2\delta_{lh}$. Thus from (2.6) we have
\begin{equation*}
\frac12\triangle|W|^2=|\nabla W|^2+2W_{ijkl}(2W_{hjkm}W_{hilm}
-\frac 12W_{ijhm}W_{klhm})+\frac{R}{2}|W|^2.
\end{equation*}
\end{remark}
\begin{proof}
By Ricci identities and (2.4), we obtain
\begin{equation*}
\begin{split}
\triangle|W|^2=&2|\nabla W|^2+2\langle W, \triangle W\rangle=2|\nabla W|^2+2W_{ijkl}W_{ijkl,mm}\\
=&2|\nabla W|^2+2W_{ijkl}(W_{ijkm,lm}+W_{ijml,km})\\
=&2|\nabla W|^2+4W_{ijkl}W_{ijkm,lm}\\
=&2|\nabla W|^2+4W_{ijkl}(W_{ijkm,ml}
+W_{hjkm}R_{hilm}+W_{ihkm}R_{hjlm}+W_{ijhm}R_{hklm}
+W_{ijkh}R_{hmlm})\\
=&2|\nabla W|^2+4W_{ijkl}(W_{hjkm}R_{hilm}
+W_{ihkm}R_{hjlm}+W_{ijhm}R_{hklm}
+W_{ijkh}R_{hmlm})\\
=&2|\nabla W|^2+4W_{ijkl}(W_{hjkm}W_{hilm}
+W_{ihkm}W_{hjlm}
+W_{ijhm}W_{hklm}
+W_{ijkh}W_{hmlm})\\
&
+\frac{4}{n-2}W_{ijkl}[W_{hjkm}(\mathring{R}_{hl}\delta_{im}-\mathring{R}_{hm}\delta_{il}+\mathring{R}_{im}\delta_{hl}-\mathring{R}_{il}\delta_{hm})\\
&+W_{ihkm}(\mathring{R}_{hl}\delta_{jm}-\mathring{R}_{hm}\delta_{jl}+\mathring{R}_{jm}\delta_{hl}-\mathring{R}_{jl}\delta_{hm})+W_{ijhm}(\mathring{R}_{hl}\delta_{km}-\mathring{R}_{hm}\delta_{kl}+\mathring{R}_{km}\delta_{hl}-\mathring{R}_{kl}\delta_{hm})\\
&+W_{ijkh}(\mathring{R}_{hl}\delta_{mm}-\mathring{R}_{hm}\delta_{ml}+\mathring{R}_{mm}\delta_{hl}-\mathring{R}_{ml}\delta_{hm})]
+\frac{4R}{n(n-1)}W_{ijkl}(W_{ljki}+W_{ilkj}+W_{ijlk})
+\frac{4R}{n}|W|^2\\
=&2|\nabla W|^2+4W_{ijkl}(2W_{hjkm}W_{hilm}
-\frac 12W_{ijhm}W_{klhm})+\frac{4R}{n}|W|^2+4W_{ijkl}W_{ijkh}\mathring{R}_{hl}.
\end{split}
\end{equation*}
This completes the proof of this Lemma.
\end{proof}

\begin{lemma}
Let $M^n(n\geq 4)$  be an $n$-dimensional complete  Riemannian manifold
with harmonic Weyl curvature.      Then
\begin{equation}
\frac12\triangle|W|^2\geq|\nabla W|^2-C(n)|W|^3-2\sqrt{\frac{n-1}{n}}|W|^2|\mathring{Ric}|+\frac{2R}{n}|W|^2,
\end{equation}
where $C(n)$ is defined in Theorem 1.3.
\end{lemma}
\begin{remark}  When $n=4$,  by Remark 2.3,  we have
\begin{equation*}
\frac12\triangle|W|^2\geq|\nabla W|^2-\frac{\sqrt{6}}{2}|W|^3+\frac{R}{2}|W|^2.
\end{equation*}
Based on the above inequality, we obtain a integral rigidity for compact $4$-manifold with harmonic Weyl curvature (see \cite{F}).
\end{remark}
\begin{proof}
By the algebraic inequality for $m$-trace-free symmetric two-tensors $T$, i.e., the eigenvalues $\lambda_i$ of $T$ satisfy $|\lambda_i|\leq\sqrt{\frac{m-1}{m}}|T|$ in \cite{Hu},  from (2.6) we get
\begin{equation}
\frac12\triangle|W|^2\geq|\nabla W|^2-2(2W_{ijlk}W_{jhkm}W_{himl}
+\frac 12W_{ijkl}W_{hmij}W_{klhm})+\frac{2R}{n}|W|^2-2\sqrt{\frac{n-1}{n}}|W|^2|\mathring{Ric}|.\end{equation}

Case 1. When $n=4$, it was proved in \cite{Hu} that $$|2W_{ijlk}W_{jhkm}W_{himl}
+\frac 12W_{ijkl}W_{hmij}W_{klhm}|\leq\frac {\sqrt{6}}{4}|W|^3.$$

Case 2. When $n=5$,  Jack and  Parker \cite{JP} have proved that $W_{ijkl}W_{hmij}W_{klhm}=2W_{ijlk}W_{jhkm}W_{himl}$. By the algebraic inequality for $m$-trace-free symmetric two-tensors $T$, i.e., $tr(T^3)\leq\frac{m-2}{\sqrt{m(m-1)}}|T|^3$ and equality holding if and only if $T$ can be diagonalized  with
$(n-1)$-eigenvalues equal to $\lambda$  in \cite{Hu}, we consider $W$ as a self-adjoint operator on $\wedge^2 V$,  and obtain
\begin{eqnarray*}|2W_{ijlk}W_{jhkm}W_{himl}
+\frac 12W_{ijkl}W_{hmij}W_{klhm}|=\frac32|W_{ijkl}W_{hmij}W_{klhm}|\leq\frac {4}{\sqrt{10}}|W|^3.\end{eqnarray*}

Case 3. When $n\geq 6$, considering $W$ as a self-adjoint operator on $S^2 V$ and $\wedge^2 V$, making use of the inequality proved by Li and Zhao \cite{LZ}  (see also \cite{Z})
and the Huisken's inequality in front, we have
\begin{eqnarray*}|2W_{ijlk}W_{jhkm}W_{himl}
+\frac 12W_{ijkl}W_{hmij}W_{klhm}|\leq2|W_{ijlk}W_{jhkm}W_{himl}|+\frac 12|W_{ijkl}W_{hmij}W_{klhm}|\\
\leq[\frac{(n-2)}{\sqrt{n(n-1)}}+ \frac{n^2-n-4}{2\sqrt{(n-2)(n-1)n(n+1)}}]|W|^3.\end{eqnarray*}

From (2.8) and Cases 1, 2 and 3, we complete the proof of this Lemma.
\end{proof}
\begin{lemma}
Let $M^n(n\geq 4)$  be an $n$-dimensional complete  Riemannian manifold. The following estimate holds
\begin{eqnarray*}
\left|-W_{ijkl}\mathring{R}_{ik}\mathring{R}_{jl}+\frac{n}{n-2}\mathring{R}_{ij}\mathring{R}_{jk}\mathring{R}_{ki}\right|\leq
\sqrt{\frac{n-2}{2(n-1)}}\left(|W|^2+\frac{2n}{n-2}|\mathring{Ric}|^2\right)^{\frac12}|\mathring{Ric}|^2.
\end{eqnarray*}
\end{lemma}
\begin{remark}
We follow these proofs of Proposition 2.1 in \cite{Ca2} and Lemma 4.7 in \cite{B2} to  prove this lemma which was proved in \cite{F2}. For completeness, we also write it out.
\end{remark}
\begin{proof}
First, we have
$$(\mathring{Ric}\circledwedge g)_{ijkl}=\mathring{R}_{ik}g_{jl}-\mathring{R}_{il}g_{jk}+\mathring{R}_{jl}g_{ik}-\mathring{R}_{jk}g_{il},$$
$$(\mathring{Ric}\circledwedge \mathring{Ric})_{ijkl}=2(\mathring{R}_{ik}\mathring{R}_{jl}-\mathring{R}_{il}\mathring{R}_{jk}),$$
where $\circledwedge$ denotes the Kulkarni-Nomizu product.
An easy computation shows
$$W_{ijkl}\mathring{R}_{ik}\mathring{R}_{jl}=\frac{1}{4}W_{ijkl}(\mathring{Ric}\circledwedge \mathring{Ric})_{ijkl},$$
$$\mathring{R}_{ij}\mathring{R}_{jk}\mathring{R}_{ik}=-\frac{1}{8}(\mathring{Ric}\circledwedge g)_{ijkl}(\mathring{Ric}\circledwedge \mathring{Ric})_{ijkl}.$$
Hence we get the following equation
\begin{equation}
-W_{ijkl}\mathring{R}_{ik}\mathring{R}_{jl}+\frac{n}{(n-2)}\mathring{R}_{ij}\mathring{R}_{jk}\mathring{R}_{ik}=-\frac{1}{4}\left(W+\frac{n}{2(n-2)}\mathring{Ric}\circledwedge g\right)_{ijkl}(\mathring{Ric}\circledwedge \mathring{Ric})_{ijkl}.
\end{equation}
Since $\mathring{Ric}\circledwedge \mathring{Ric}$ has the same symmetries with the Riemannian curvature tensor, it can be orthogonally decomposed as
$$\mathring{Ric}\circledwedge \mathring{Ric}=T+V'+U'.$$
Here $T$ is totally trace-free, and
$$V'_{ijkl}=-\frac{2}{n-2}\left(\mathring{Ric}^2\circledwedge g\right)_{ijkl}+\frac{2}{n(n-2)}|\mathring{Ric}|^2(g\circledwedge g)_{ijkl},$$
$$U'_{ijkl}=-\frac{1}{n(n-1)}|\mathring{Ric}|^2(g\circledwedge g)_{ijkl},$$
where $\left(\mathring{Ric}^2\right)_{ik}=\mathring{R}_{ip}\mathring{R}_{kp}$. Taking the squared norm we obtain
$$|\mathring{Ric}\circledwedge \mathring{Ric}|^2=8|\mathring{Ric}|^4-8|\mathring{Ric}^2|^2,$$
$$|V'|^2=\frac{16}{n-2}|\mathring{Ric}^2|^2-\frac{16}{n(n-2)}|\mathring{Ric}|^4,$$
$$|U'|^2=\frac{8}{n(n-1)}|\mathring{Ric}|^4.$$
In particular, one has
\begin{equation}
|T|^2+\frac{n}{2}|V'|^2=|\mathring{Ric}\circledwedge \mathring{Ric}|^2+\frac{n-2}{2}|V'|^2-|U'|^2=\frac{8(n-2)}{n-1}|\mathring{Ric}|^4.\label{77777}
\end{equation}

We now estimate the right hand side of (2.9). Using  Cauchy-Schwarz inequality, \eqref{77777} and  the fact that $W$ and $T$ are totally trace-free we obtain
\begin{eqnarray*}
\left|\left(W+\frac{n}{2(n-2)}\mathring{Ric}\circledwedge g\right)_{ijkl}(\mathring{Ric}\circledwedge \mathring{Ric})_{ijkl}\right|^2
&=&\left|\left(W+\frac{n}{2(n-2)}\mathring{Ric}\circledwedge g\right)_{ijkl}(T+V')_{ijkl}\right|^2\\
&=&\left|\left(W+\frac{\sqrt{2n}}{2(n-2)}\mathring{Ric}\circledwedge g\right)_{ijkl}\left(T+\sqrt{\frac{n}{2}}V'\right)_{ijkl}\right|^2\\
&\leq&\left|W+\frac{\sqrt{2n}}{2(n-2)}\mathring{Ric}\circledwedge g\right|^2\left(|T|^2+\frac{n}{2}|V'|^2\right)\\
&=&\frac{8(n-2)}{n-1}|\mathring{Ric}|^4\left(|W|^2+\frac{2n}{n-2}|\mathring{Ric}|^2\right).
\end{eqnarray*}
This estimate together with (2.9) concludes this proof.
\end{proof}
\section{Proofs of Theorems}
In this section,  we  prove our main theorems and corollaries by using lemmas in Section 2.
\begin{proposition}\label{99999}
Let $(M^n, g)(n\geq4)$  be an $n$-dimensional   compact Riemannian manifold  with harmonic Weyl curvature and positive scalar curvature. If
\begin{equation}\int_{M^n}|\nabla \mathring{Ric}|^2\geq\frac{(n-2)^2}{4n(n-1)}\int_{M^n}|\nabla R|^2\end{equation}
and
\begin{equation}
|W|^2+\frac{2n}{n-2}|\mathring{Ric}|^{2}
<\frac{2}{{(n-2)(n-1)}}R^2,
\end{equation}
then $M^n$ is an Einstein manifold.
In particular, if the pinching condition in (3.2) is weakened to \begin{equation}
|W|^2+\frac{2n}{n-2}|\mathring{Ric}|^{2}
<\frac{4}{{n^2C(n)^2}}R^2,
\end{equation} then $M^n$ is isometric to a quotient of  $\mathbb{S}^n$.
\end{proposition}
\begin{proof}
Integrating (2.5) by parts over $M^n$ and using (2.2) we get
\begin{equation}
\begin{split}
0&=\int_{M^n}|\nabla \mathring{Ric}|^2-\frac{n-2}{2(n-1)}\int_{M^n}\mathring{R}_{ij,j}R_{,i}+\int_{M^n} W_{kijl}\mathring{R}_{ij}\mathring{R}_{kl}+\frac{n}{n-2}\int_{M^n}\mathring{R}_{ij}\mathring{R}_{jl}\mathring{R}_{li}
+\frac{1}{n-1}\int_{M^n} R|\mathring{Ric}|^2\\
&=\int_{M^n}|\nabla \mathring{Ric}|^2-\frac{(n-2)^2}{4n(n-1)}\int_{M^n}|\nabla R|^2+\int_{M^n} \left(-W_{ikjl}\mathring{R}_{ij}\mathring{R}_{kl}+\frac{n}{n-2}\mathring{R}_{ij}\mathring{R}_{jl}\mathring{R}_{li}\right)
+\frac{1}{n-1}\int_{M^n} R|\mathring{Ric}|^2.
\end{split}
\end{equation}
Substituting (3.1) and Lemma 2.6 into (3.4) we get
\begin{equation}
\begin{split}
0&\geq-\sqrt{\frac {n-2}{2(n-1)}}\int_{M^n}\left(|W|^2+\frac{2n}{n-2}|\mathring{Ric}|^2\right)^{\frac 12}|\mathring{Ric}|^2
+\frac{1}{n-1}\int_{M^n} R|\mathring{Ric}|^2\\
&=\int_{M^n}\left[\frac{R}{n-1}-\sqrt{\frac {n-2}{2(n-1)}}\left(|W|^2+\frac{2n}{n-2}|\mathring{Ric}|^2\right)^{\frac 12}\right]|\mathring{Ric}|^2.
\end{split}
\end{equation}
Hence the pinching condition (3.2) implies that $\mathring{Ric}=0$, i.e., $M^n$ is Einstein.

If (3.3) holds,  then $M^n$ is Einstein. Thus it is easy to see that
\begin{equation}
\frac12\triangle|W|^2\geq|\nabla W|^2-C(n)|W|^3+\frac{2R}{n}|W|^2.
\end{equation}
If $|W|<\frac{2}{nC(n)}R$, from (3.6) we have  $|W|$ is a  subharmonic function, and $|W|$ is constant. Hence from (3.6) we know that $W=0$, i.e., $M^n$ is locally conformally flat. By the  rigidity result for positively curved Einstein manifolds (see Theorem 1.1 in \cite{FX2}),
we can directly get that $M^n$ is isometric to a quotient of the round $\mathbb{S}^n$.

Then we complete the proof of Propositon \ref{99999}.
\end{proof}

\begin{proof}[{\bf Proofs of Theorems
1.1 and 1.3}]We compute
  \begin{eqnarray}
|\nabla A|^2&=&|\nabla Ric|^2-\frac{1}{n-1}|\nabla R|^2+\frac{|\nabla R|^2}{4(n-1)^2}n\\\nonumber
&=&|\nabla Ric|^2-\frac{3n-4}{4(n-1)^2}|\nabla R|^2\\\nonumber
&=&|\nabla \mathring{Ric}|^2+\frac{(n-2)^2}{4n(n-1)^2}|\nabla R|^2,
  \end{eqnarray}
and
\begin{equation}
|\nabla trA|^2=\frac{(n-2)^2}{4(n-1)^2}|\nabla R|^2.
\end{equation}
Since $\sigma_2(A_g)$ is a
positive constant,  the  inequality of Kato type due to Li  and Simon\cite{{HLS},{L},{SI}}, i.e.,
\begin{equation}|\nabla A|^2\geq|\nabla trA|^2\end{equation}
holds. From (3.7) and (3.8), (3.9) implies that
\begin{equation*}\int_{M^n}|\nabla \mathring{Ric}|^2\geq\frac{(n-2)^2}{4n(n-1)}\int_{M^n}|\nabla R|^2.\end{equation*}
By (1.1), the pinching conditions (1.2) and (1.3) are equivalent to (3.2) and (3.3), respectively. By Proposition 3.1, we can complete the proofs of Theorems 1.1 and 1.3.
\end{proof}
\begin{remark}
Since  Riemannian manifolds  with harmonic curvature has constant scalar curvature, (3.1) naturally holds. Proposition 3.1 improves Theorem 1.1 in \cite{FXT}.
\end{remark}

\begin{proof}[\bf Proof of Corollary 1.4]
When $n=4$, the pinching condition (1.2) in Theorem 1.1 is reduced to  (1.4) in Corollary 1.4. By Theorem 1.1, $M^4$ is Einstein. Thus by (1.1), (1.4) is equivalent to
\begin{equation*}
|W|^2<\frac{R^2}{3}.
\end{equation*}
By Theorem 1.8 in \cite{F} (see also \cite{F2}), we obtain that
$M^4$ is isometric to a quotient of the round $\mathbb{S}^4$ or a  $\mathbb{CP}^2$ with the Fubini-Study metric. Then we complete the proof of Corollary \ref{66666}.
\end{proof}

\begin{proof}[{\bf Proofs of Corollary
1.7}]According to  Theorem 1.1 in \cite{CGY}, there exists a conformal metric $\tilde{g}$ such that  $\sigma_2(A_{\tilde{g}})$ is a
positive constant and $R_{\tilde{g}}$ is positive. Hence by Corollary 1.5, we can complete the proof of Corollary 1.7.
\end{proof}
\begin{proof}[{\bf Proof of Corollary
1.9}]
When $M^n$ is locally conformally flat, the pinching condition (1.2) in Theorem 1.1 is reduced to (1.5) in Corollary 1.9.
 By Theorem 1.1, $M^n$  is isometric to a quotient of the round $\mathbb{S}^n.$ Then we complete the proof of Corollary \ref{66666666}.
\end{proof}

\begin{proof}[{\bf Proof of Theorem
1.11}]
By the Kato inequality $|\nabla W|^2\geq \frac{n+1}{n-1}|\nabla |W||^2$ (see \cite{Br}) and (2.7), we have
\begin{equation}
\frac12\triangle|W|^2\geq\frac{n+1}{n-1}|\nabla |W||^2-C(n)|W|^3-2\sqrt{\frac{n-1}{n}}|W|^2|\mathring{Ric}|+\frac{2R}{n}|W|^2.
\end{equation}
We rewrite (3.10) as
\begin{equation}
|W|\triangle|W|\geq\frac{2}{n-1}|\nabla |W||^2-C(n)|W|^3-2\sqrt{\frac{n-1}{n}}|W|^2|\mathring{Ric}|+\frac{2R}{n}|W|^2
\end{equation}
in the sense of distributions.
Setting $u=|W|$, we compute from (3.11)
\begin{equation}
\begin{split}
u^{\gamma}\triangle u^{\gamma}&=u^{\gamma}\left(\gamma(\gamma-1)u^{\gamma-2}|\nabla u|^2+\gamma u^{\gamma-1}\triangle u\right)\\
&=\frac{\gamma-1}{\gamma}|\nabla u^{\gamma}|^2+\gamma
u^{2\gamma-2}u\triangle u\\
&\geq\left(1-\frac{n-3}{(n-1)\gamma}\right)|\nabla u^{\gamma}|^2-C(n)\gamma u^{2\gamma+1}-2\sqrt{\frac{n-1}{n}}\gamma u^{2\gamma}|\mathring{Ric}|+\frac{2R}{n}\gamma u^{2\gamma}.
\end{split}
\end{equation}
Integrating (3.12) by parts over $M^n$, it follows that
\begin{equation}
0\geq\left(2-\frac{n-3}{(n-1)\gamma}\right)\int_{M^n} |\nabla u^{\gamma}|^2-C(n)\gamma \int_{M^n} u^{2\gamma+1}-2\sqrt{\frac{n-1}{n}}\gamma\int_{M^n} u^{2\gamma}|\mathring{Ric}|+\frac{2}{n}\gamma\int_{M^n} R u^{2\gamma}.
\end{equation}
For $2-\frac{n-3}{(n-1)\gamma}>0$, by the definition of Yamabe constant and (3.13), we can obtain
\begin{equation}
\begin{split}
0\geq\left[(2-\frac{n-3}{(n-1)\gamma})\frac{n-2}{4(n-1)}Y(M^n,[g])-C_2(n)\gamma\left(\int_{M^n} |W|^\frac{n}{2}\right)^\frac{2}{n}-2\sqrt{\frac{n-1}{n}}\gamma\left(\int_{M^n}|\mathring{Ric}|^\frac{n}{2}\right)^{\frac 2n}\right]
\left(\int_{M^n}  |W|^{\frac{2n\gamma}{n-2}}\right)^{\frac{n-2}{n}}\\
+\frac {8(n-1)\gamma+\frac{n(n-2)(n-3)}{(n-1)\gamma}-2n(n-2)}{4n(n-1)}\int_{M^n} R|W|^{2\gamma}.
\end{split}
\end{equation}

Case 1. If $n=5$,  choose $\gamma=\frac12$. It follows from (3.14) that
\begin{equation}
\begin{split}
0\geq\left[\frac{3}{8}Y(M^n,[g])-C_2(5)\left(\int_{M^5} |W|^\frac{5}{2}\right)^\frac{2}{5}-2\sqrt{\frac{4}{5}}\left(\int_{M^5}|\mathring{Ric}|^{\frac{5}{2}}\right)^{\frac 25}\right]
\left(\int_{M^5}  |W|^{\frac{5}{3}}\right)^{\frac{3}{5}}\\
+\frac {1}{80}\int_{M^5} R|W|.
\end{split}
\end{equation}
From (3.15), the pinching condition (1.6) implies that $M^5$ is locally conformal flat. Moreover,  we get the same conclusion if we assume just the weak inequality in (1.6).

Case 2. If $n\neq 5$,  choose $\frac 1\gamma=\frac{n-1}{n-3}(1+\sqrt{1-\frac {8(n-3)}{(n-2)n}})$. From (3.14), we get
\begin{equation}
0\geq\left[\frac{2}{n}Y(M^n,[g])-C_2(n)\left(\int_{M^n} |W|^\frac{n}{2}\right)^\frac{2}{n}-2\sqrt{\frac{n-1}{n}}\left(\int_{M^n}|\mathring{Ric}|^{\frac{n}{2}}\right)^{\frac 2n}\right]
\left(\int_{M^n}  |W|^{\frac{2n\gamma}{n-2}}\right)^{\frac{n-2}{n}}.
\end{equation}
From (3.16), the pinching condition (1.6) implies that $M^n$ is locally conformal flat. Then we complete the proof of Theorem \ref{5555555}.
\end{proof}

\bibliographystyle{amsplain}

\end{document}